\documentclass[a4paper,12pt]{article}
\usepackage[utf8]{inputenc}
\usepackage{t1enc}
\usepackage{amsmath,amssymb}
\usepackage{amsfonts}
\usepackage{array}
\PassOptionsToPackage{hyphens}{url}
\usepackage{bbm,xurl}
\usepackage[unicode,colorlinks]{hyperref}
\usepackage{amsthm,comment}
\usepackage[hmargin=1in,vmargin=1in]{geometry}
\usepackage{tikz}
\hypersetup{colorlinks=true,breaklinks=true}
\numberwithin{equation}{section}

\newtheorem{theorem}{Theorem}[section]
\newtheorem{proposition}[theorem]{Proposition}
\newtheorem{lemma}[theorem]{Lemma}

\theoremstyle{definition}

\newcommand{\wt}{\widetilde}
\renewcommand{\P}{\mathbf P}
\newcommand{\E}{\mathbf E}
\renewcommand{\d}{\mathrm d}
\newcommand{\R}{\mathbb R}
\newcommand{\Z}{\mathbb Z}
\renewcommand{\O}{\mathcal O}
\newcommand{\ind}{\mathbbm 1}

\DeclareMathOperator*{\argmax}{argmax}

\author{Bori Anna M\'esz\'aros\thanks{Department of Stochastics, Institute of Mathematics,
Budapest University of Technology and Economics, M\H uegyetem rkp.\ 3., H-1111 Budapest, Hungary.
E-mail: {\tt meszaros.bori@edu.bme.hu}}
\and
B\'alint Vet\H o\thanks{Department of Stochastics, Institute of Mathematics,
Budapest University of Technology and Economics, M\H uegyetem rkp.\ 3., H-1111 Budapest, Hungary
and HUN--REN Alfr\'ed R\'enyi Institute of Mathematics, Re\'altanoda u.\ 13--15., H-1053 Budapest, Hungary.
E-mail: {\tt vetob@math.bme.hu}}
}

\title{Stabilization time of finite configurations with a second class particle in discrete TASEP}

\begin{document}

\maketitle

\begin{abstract}
We consider finite configurations of particles and holes sampled according to Bernoulli product measure
and with a second class particle added to a random position.
The stabilization time is the number of steps needed to reach an ordered state under discrete time TASEP dynamics with parallel update.
We describe the additional time of stabilization caused by the presence of the second class particle.
\end{abstract}

\section{Introduction}

The totally asymmetric simple exclusion process (TASEP) is an interacting particle system
with configurations on the integer lattice $\Z$ where each lattice point is either occupied or vacant.
It was first introduced in continuous time in \cite{S70}, the discrete time version was studied e.g.\ in \cite{S93}.
Our interest is in the dynamics with parallel update
where particles jump to the right by one independently with a fixed probability $q$ if the target position is empty.
See also \cite{MS11} for other possible discrete time TASEP dynamics.

The particle dynamics can be described in terms of the height function determined modulo a constant global shift
by particles and holes corresponding to segments with slope $-1$ and $+1$ respectively.
Particle jumps mean that pairs of segments with slope $-1$ on the left and with slope $+1$ on the right
are swapped with probability $q$ turning the local maximum of the height into a local minimum.

The TASEP evolution on finite segments of the lattice corresponds to the swap dynamics of finite sequences of two types.
The stabilization time of a string is the number of steps from the initial configuration to the final state
under the swap dynamics with $q=1$ and with parallel update.
The finiteness of the stabilization time for any initial string is an exercise in \cite{UM96}
interpreted in terms of a line of soldiers facing in random directions.
For product Bernoulli initial condition with density $p\in[0,1]$,
the asymptotic distribution of the stabilization time was described in \cite{FNN15}
as the length of the string tends to infinity, see Theorem \ref{thm:twotypes} below.

In the present paper, we study the stabilization time of random sequences with three types.
In addition to particles and holes, the third type is the second class particle which was defined for TASEP in \cite{FK95} and served as the main tool to bound the current fluctuations and diffusivity in ASEP in \cite{BS10}.
Parallel update is not a priori well-defined with second class particles because of possible conflicts.
We resolve the conflicts by the convention that the dynamics for three types reduces to that for two types
under the projection which replaces second class particles by holes.
We compare the full stabilization time for three types to that for the projected dynamics
in the presence of one second class particle.
Started from product Bernoulli initial strings together with the location of the second class particle,
we describe the asymptotic distribution of the additional time in the limit for long strings.

\subsection{Stabilization with two types}

We introduce the sorting model for binary strings that corresponds to discrete time TASEP with parallel update,
and we recall the result on the stabilization time of random strings from \cite{FNN15}.
Let $\Omega_n^{(2)}=\{0,2\}^n$ be the set of strings of length $n$ with two types.
We define the evolution step $Z^{(2)}:\Omega_n^{(2)}\to\Omega_n^{(2)}$ in a way that it swaps neighouring elements which are in increasing order,
that is, every occurrence of $02$ is replaced by $20$ within the string.
Replacing each $0$ with a particle and each $2$ with a hole, $Z^{(2)}$ describes a discrete time step of TASEP with parallel update.

Since no two occurrences of $02$ can overlap, the definition does not result in conflicts between overlapping pairs.
The successive application of $Z^{(2)}$ eventually sorts the elements in non-increasing order, and the process stabilizes in at most $n-1$ steps.
Let
\begin{equation}
T_n^{(2)}(\omega)=\min\{k\ge0 : (Z^{(2)})^k(\omega)=(Z^{(2)})^{k+1}(\omega)\}
\end{equation}
denote the stabilization time for the binary sequence $\omega\in\Omega_n^{(2)}$ where $(Z^{(2)})^k$ is the $k$fold application of $Z^{(2)}$.

The natural measure on $\Omega_n^{(2)}$ is the Bernoulli product measure
where each element of the string is $2$ with probability $p$ and $0$ with probability $1-p$ independently.
The main result in \cite{FNN15} identifies the asymptotic fluctuations of the stabilization time $T_n^{(2)}$ as $n\to\infty$ depending on $p$.

\begin{theorem}\label{thm:twotypes}
For $p\in(0,1)\setminus\{1/2\}$ fixed, we have that
\begin{equation}\label{Tn2p>1/2}
\frac{T_n^{(2)}-\max(p,1-p)n}{\sqrt n}\stackrel\d\Longrightarrow\mathcal N(0,p(1-p))
\end{equation}
in distribution as $n\to\infty$
where the limit is Gaussian with mean $0$ and variance $p(1-p)$.

For $p=1/2$,
\begin{equation}\label{Tn2p=1/2}
\frac{T_n^{(2)}-n/2}{\sqrt n}\stackrel\d\Longrightarrow\frac{\chi_3}2
\end{equation}
in distribution as $n\to\infty$ where $\chi_3$ is the chi distribution with parameter $3$.

For $p=1/2+\lambda/(2\sqrt n)$ with $\lambda\in\R$ fixed, we have
\begin{equation}
\frac{T_n^{(2)}-n/2}{\sqrt n}\stackrel\d\Longrightarrow M_1^\lambda-\frac12B_1^\lambda
\end{equation}
in distribution as $n\to\infty$ where $B_t^\lambda$ is a Brownian motion with drift $\lambda$
and $M_t^\lambda=\max_{s\in[0,t]}B_s^\lambda$ is its running maximum process.
\end{theorem}

\subsection{Stabilization model with three types}

The purpose of this paper is to investigate the case of three types.
Let $\Omega_n=\{0,1,2\}^n$ denote the space of strings of length $n$ of three types.
The evolution step $Z:\Omega_n\to\Omega_n$ would ideally swap all possible increasing substrings $01,02,12$
into their decreasing counterparts.
However, these substrings can overlap and form the substring $012$.
To resolve the conflict, one has to give priority to one of the two overlapping substrings.
In this paper, we use the convention that the substring $12$ has priority over $01$ and changes first into $21$, that is,
\begin{equation}\label{priority}
Z(\dots012\dots)=\dots021\dots.
\end{equation}
Having this settled, the evolution step $Z$ is well-defined.

In terms of a particle system, $0$ corresponds to a particle, $2$ is a vacant site, and $1$ is a second class particle.
Under the TASEP dynamics, the particles jump to the right if their right neighbour is vacant.
Second class particles also jump to the right to vacant locations but they can also jump to the left if there is a particle on their left.
The two possible jumps of second class particles can lead to a conflict in discrete time with parallel update
if three consecutive sites contain a particle, a second class particle and a hole.
We resolve this conflict by giving priority to the right jumps of the second class particle which is the priority rule given in \eqref{priority}.

Formally, if $\omega_i=0,\omega_{i+1}=1,\omega_{i+2}=2$ for some $i=1,\dots,n-2$ then $(Z\omega)_i=0,(Z\omega)_{i+1}=2,(Z\omega)_{i+2}=1$.
Furthermore, if $\omega_{i-1}\ge\omega_i<\omega_{i+1}\ge\omega_{i+2}$ for some $i=1,\dots,n-1$
(with the inequality always considered satisfied for $\omega_0$ and $\omega_{n+1}$) then $(Z\omega)_i=\omega_{i+1},(Z\omega)_{i+1}=\omega_i$.

The two natural projections $\Pi_1,\Pi_2:\Omega_n\to\Omega_n^{(2)}$ are given by
\begin{equation}
(\Pi_i\omega)_j=\pi_i(\omega_j)
\end{equation}
for $i=1,2$ and $j=1,\dots,n$ and $\omega\in\Omega_n$
where $\pi_1(0)=\pi_1(1)=0$, $\pi_1(2)=2$ and $\pi_2(0)=0$, $\pi_2(1)=\pi_2(2)=2$.
In other words, the two types under $\Pi_1$ are $\{0,1\}$ and $\{2\}$, that is, we consider $0$ and $1$ identical.
Under $\Pi_2$, the types are $\{0\}$ and $\{1,2\}$.
There is no evolution rule for three types under which both natural projections evolve according to the dynamics for two types.
With our convention, the evolution for three types projected by $\Pi_1$ follows the dynamics of two types, that is, $\Pi_1Z=Z^{(2)}\Pi_1$ but it does not hold for $\Pi_2$ instead of $\Pi_1$.

Starting from any string $\omega\in\Omega_n$, the repeated applications of the evolution step $Z$ eventually stabilize,
and the procedure results in a string where all elements are sorted in non-increasing order.
The stabilization time $T_n(\omega)$ is the number of steps needed to achieve the non-increasing order starting from $\omega$.
Formally, we let
\begin{equation}
T_n(\omega)= \min\{k\ge0 : Z^k(\omega)=Z^{k+1}(\omega)\}.
\end{equation}
In the example
\begin{equation}
0122102 \xrightarrow{Z} 0212120 \xrightarrow{Z} 2021210 \xrightarrow{Z} 2202110 \xrightarrow{Z} 2220110 \xrightarrow{Z} 2221010\xrightarrow{Z} 222110
\end{equation}
the stabilization time is $6$, that is, $T_7(0122102)=6$.

In this paper, we consider a special initial distribution of the string $\omega\in\Omega_n$ with three types
where there is a single element of type $1$ in $\omega$.
The initial string $\omega$ is constructed in two steps as follows.
First, we sample the locations of $2$s according to the Bernoulli product measure of parameter $p$,
that is, we let $\omega'\in\Omega_n$ be a string where each element is independently equal to $2$ with probability $p$ and is $0$ with probability $1-p$.
Then conditionally given the set $A(\omega')$ of locations of $0$s in $\omega'$, we sample $U$ with uniform distribution on the set $A(\omega')$.
Then we change the corresponding element $\omega'_U$ into $1$ to obtain $\omega$.
If $A(\omega')=\emptyset$, then we let $\omega_1=1$.
We refer to the distribution of $\omega$ as the Bernoulli initial condition of parameter $p$ with a second class particle in uniform position.

\subsection{Main results}

The stabilization time result in Theorem \ref{thm:twotypes} describes the asymptotic behaviour of $T_n^{(2)}(\Pi_1(\omega))$
for any string $\omega\in\Omega_n$ with three types.
After $T_n^{(2)}(\Pi_1(\omega))$ steps in a string $\omega\in\Omega_n$, all the $2$s are sorted to the beginning of the string
but there might still be some $0$s before the position of $1$ in the string, which need to be swapped with $1$.
Hence $T_n^{(2)}(\Pi_1(\omega))$ is not equal to the total stabilization time $T_n(\omega)$.
We denote by $E_n$ the number of excess steps, formally, we let
\begin{equation}
E_n(\omega)=T_n(\omega)-T_n^{(2)}(\Pi_1(\omega)).
\end{equation}
Our first main result is the following characterization of the excess for fixed $p$.

\begin{theorem}\label{thm:fixp}
Let $p\in(0,1)$ be fixed and let $\omega\in\Omega_n$ be sampled according to the Bernoulli initial condition of parameter $p$
with a second class particle in uniform position.
Then the excess $E_n$ converges in distribution
\begin{equation}
E_n\stackrel\d\Longrightarrow\left\{\begin{array}{ll}1&\mbox{if }p>1/2,\\X&\mbox{if }p=1/2,\\0&\mbox{if }p<1/2\end{array}\right.
\end{equation}
as $n\to\infty$ where the distribution of $X$ is given by $\P(X=0)=\P(X=1)=1/2$.
\end{theorem}

The second main result describes the limiting distribution of the excess $E_n$ in the case when $p$ is critically scaled with $n$ around $1/2$.

\begin{theorem}\label{thm:crit}
Let $p=\frac{1}{2}-\frac{\lambda}{2\sqrt{n}}$ for some $\lambda\in\R$ be fixed
and let $\omega\in\Omega_n$ be sampled according to the Bernoulli initial condition of parameter $p$ with a second class particle in uniform position.
Then we have that
\begin{equation}
E_n\stackrel\d\Longrightarrow X^{(\lambda)}
\end{equation}
as $n\to\infty$ where the distribution of $X^{(\lambda)}$ is given by
\begin{equation}
\P(X^{(\lambda)}=0)=1-\P(X^{(\lambda)}=1)=\E(\argmax_{s\in[0,1]}B_s^{(\lambda)})
\end{equation}
with a Brownian motion $B_s^{(\lambda)}$ with drift $\lambda$.
\end{theorem}

The paper is organized as follows.
We characterize the possible initial positions of the $1$ in a string which result in a positive excess in Section \ref{s:switching}
by providing a structural understanding of the evolution of heights.
The main results are proved in Section \ref{s:asympexcess} using hitting time bounds for random walks.
We prove these bounds in Section \ref{s:hitting}.

\paragraph{Acknowledgments.}
The work of B.\ A.\ M\'esz\'aros was supported by the NKFI (National Research, Development and Innovation Office)
grant FK142124.
The work of B.\ Vet\H o was supported by the NKFI grants ADVANCED 150474 and KKP144059
``Fractal geometry and applications''.

\section{Switching position for the excess}
\label{s:switching}

The proofs of Theorems \ref{thm:fixp} and \ref{thm:crit} are based on understanding
how the initial position of the element $1$ determines the excess.
In particular, in this section, we give a structural understanding of the evolution of local maxima
which leads to a characterization of the initial positions of $1$ with excess.

For a string $\omega$ with a unique $1$, the maximal prefix of $2$s and the maximal suffix of $0$s are
\begin{equation}\begin{aligned}
L(\omega)&=\max\{k:(\Pi_1\omega)_1=(\Pi_1\omega)_2=\dots=(\Pi_1\omega)_k=2\},\\
R(\omega)&=\max\{k:(\Pi_1\omega)_n=(\Pi_1\omega)_{n-1}=\dots=(\Pi_1\omega)_{n-k+1}=0\}.
\end{aligned}\end{equation}
The height function of a string with two types consists of $\pm1$ steps.
For a string $\omega\in\Omega_n$ with three types, we consider the height function corresponding to the projection $\Pi_1(\omega)$ given by
\begin{equation}
S_k=\sum_{i=1}^k(1-(\Pi_1\omega)_i)
\end{equation}
for $k=0,1,\dots,n$.

We introduce
\begin{equation}\label{defK}
K(\omega)=\max\{k\in\{L+2,\dots,n-R-1\}: S_l<S_k-1,l=L,\dots,k-2\}
\end{equation}
as the rightmost position of the string $\omega$ with the property that the height remains strictly below $S_K-1$ between $L$ and $K-2$ and we let $K=L+2$ if the set on the right-hand side of \eqref{defK} is empty.
See Figure \ref{fig:KM} for an illustration.
By definition \eqref{defK}, we have $\pi_1(\omega_{K-2})=\pi_1(\omega_{K-1})=\pi_1(\omega_K)=0$ and $\pi_1(\omega_{K+1})=2$ unless $K=L+2$.

\begin{figure}
\centering
\resizebox{0.99\textwidth}{!}{%
\begin{tikzpicture}[every node/.style={font=\small}]

\draw[->] (-0.5,0) -- (20.5,0) node[right] {$k$};
\draw[->] (0,-0.4) -- (0,5.8) node[above] {$S_k$};

\draw[very thick]
  (0,0)
  -- (1,1)
  -- (2,0)
  -- (3,1)
  -- (4,0)
  -- (5,1)
  -- (6,2)
  -- (7,3)    
  -- (8,2)
  -- (9,3)
  -- (10,2)
  -- (11,3)
  -- (12,4)   
  -- (13,3)
  -- (14,4)
  -- (15,3)
  -- (16,4)
  -- (17,5)   
  -- (18,4)
  -- (19,5)
  -- (20,4); 

\foreach \x/\y in {
  0/0,1/1,2/0,3/1,4/0,5/1,6/2,7/3,8/2,9/3,10/2,11/3,
  12/4,13/3,14/4,15/3,16/4,17/5,18/4,19/5,20/4}
  \fill (\x,\y) circle (2.2pt);

\draw[dashed] (7,0) -- (7,3);
\draw[dashed] (12,0) -- (12,4);
\draw[dashed] (17,0) -- (17,5);

\draw[dashed] (0,3) -- (7,3);
\draw[dashed] (0,4) -- (12,4);
\draw[dashed] (0,5) -- (17,5);

\node[left] at (0,5) {$S_M$};
\node[left] at (0,4) {$S_M-1$};
\node[left] at (0,3) {$S_M-2$};

\node[below] at (7,-0.15) {$K=M_2$};
\node[below] at (12,-0.15) {$M_1$};
\node[below] at (17,-0.15) {$M_0=M$};

\end{tikzpicture}%
}
\caption{An example illustrating the points $M_0=M$, $M_1$, and $K=M_2$.}
\label{fig:KM}
\end{figure}
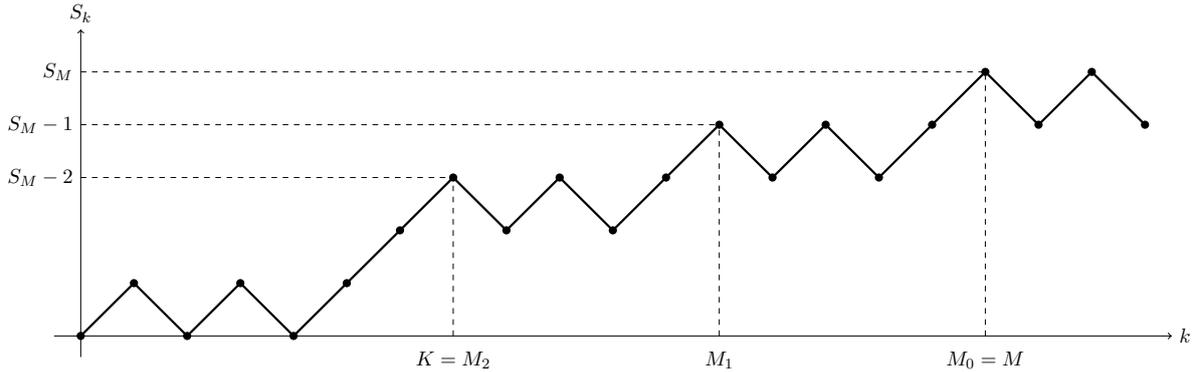

By the next result, $K$ is the \emph{switching position} for the positivity of the excess in the sense that
the excess $E_n$ is zero if and only if the initial position of the $1$ is before $K$.
In addition, the excess can be more than one only when the initial position of the $1$ is to the right of $n-R$, that is, there is no $2$ after it.

\begin{proposition}\label{prop:KU}
For any string $\omega$ with a single $1$ in it, we have the equality of events
\begin{equation}\label{KU}
\{U(\omega)<K(\omega)\}=\{E_n(\omega)=0\}.
\end{equation}
Furthermore,
\begin{equation}\label{Un-R}
\{E_n(\omega)>1\}\subseteq\{U(\omega)>n-R(\omega)\}.
\end{equation}
\end{proposition}

Let $M_k$ denote the leftmost location with height that is $k$ less than the maximum of $S_k$ on $\{L+1,\dots,n-R-1\}$, that is,
\begin{equation}\label{defMk}
M_k(\omega)=\min\{j\in\{L+1,\dots,n-R-1\}:S_j=\max\{S_i:i\in\{L+1,\dots,n-R-1\}-k\}
\end{equation}
for $k=0,1,2,\dots$ which are well-defined as long as the set of indices $j$ above is not empty.
We mention that $M_0(\omega)$ is the position of the leftmost maximum of $S_k$ for $k\in\{L+1,\dots,n-R-1\}$ and we also denote it by $M=M_0$.
See Figure \ref{fig:KM}.
It follows from the definition that $\pi_1(\omega_{M_k})=\pi_1(\omega_{M_k-1})=0$ unless $M_k=L+1$
and that $S_j<S_{M_k}$ for $j=L,\dots,M_k-1$.

For a given $\omega$, we have $K(\omega)=M_m(\omega)$ for some integer $m$ by definition \eqref{defK} unless the set of indices $k$ is empty on the right-hand side of \eqref{defK}.
It is impossible to have $M_{k-1}(\omega)=M_k(\omega)+1$ for some $k=1,\dots,m$ because it would mean that $K(\omega)\ge M_{k-1}(\omega)$.
Hence the points $M_0(\omega),M_1(\omega),\dots,M_m(\omega)$ decompose the interval $[K,M]$ into blocks of integer lengths that are all at least $3$.

The idea of the proof is that for the given $\omega$,
we keep track of the point $M_k$ for $k=0,\dots,m$ under the evolution step $Z$ applied to $\omega$ repeatedly.
We introduce
\begin{equation}
\omega^{(j)}=Z^j\omega,\quad L^{(j)}=L(Z^{j}\omega),\quad R^{(j)}=R(Z^j\omega),\quad M_k^{(j)}=M_k(Z^j\omega)
\end{equation}
and we let $U^{(j)}$ denote the position of $1$ in $\omega^{(j)}$.
Evolution occurs in three phases.
The first phase is the \emph{bulk bahaviour} of $M_k$.
For any $k=0,\dots,m$, we have
\begin{equation}\label{Mkbulk}
M_k^{(j+1)}=M_k^{(j)}-1
\end{equation}
if $M_k^{(j)}\ge L^{(j)}+2$.
Since $M_{k-1}^{(j)}\neq M_k^{(j)}+1$, $M_k^{(j)}$ is a local maximum of the height,
that is, $\pi_1(\omega_{M_k^{(j)}})=0$ and $\pi_1(\omega_{M_k^{(j)}+1})=2$.
On the other hand, all local maxima in $[L^{(j)},M_k^{(j)}-1]$ have heights  lower than $S_{M_k^{(j)}}$ by \eqref{defMk}.
Applying the evolution step $\omega\mapsto Z\omega$ decreases all the local maxima of the height,
resulting in that $M_k^{(j)}-1$ becomes a new leftmost local maximum with value $S_{M_0^{(j+1)}}-k$, that is, \eqref{Mkbulk} holds.

The second phase is the \emph{edge bahaviour} which can happen in later steps of the evolution.
If $M_k^{(j)}=L^{(j)}+1$ and $M_{k-1}^{(j)}>L^{(j)}+2$ for some $j=1,2,\dots$, then we have
\begin{equation}\label{Mkedge}
M_k^{(j+1)}=M_k^{(j)}+1.
\end{equation}
Indeed, if $M_k^{(j)}$ is a leftmost local maximum with value $S_{M_0^{(j)}}-k$
then $\pi_1((\omega^{(j)})_{M_k^{(j)}})=0$ and $\pi_1((\omega^{(j)})_{M_k^{(j)}+1})=2$ hold.
The pair $02$ is transformed into $20$, so the height at $M_k^{(j)}\pm1$ becomes $S_{M_0^{(j+1)}}-k=S_{M_0^{(j)}}-k-1$
but $M_k^{(j)}-1=L^{(j+1)}$ which proves \eqref{Mkedge}.

The final phase is the \emph{annihilation}.
If $M_k^{(j)}=L^{(j)}+1$ and $M_{k-1}^{(j)}=L^{(j)}+2$
then we necessarily have $\pi_1((\omega^{(j)})_{L^{(j)}+1})=\pi_1((\omega^{(j)})_{L^{(j)}+2})=0$ and $\pi_1((\omega^{(j)})_{L^{(j)}+3})=2$
since $M_{k-2}$ can only be in the bulk behaviour phase.
This projected pattern $002$ after the left block of $2$s is updated in the next step to $020$, resulting in having
\begin{equation}
M_{k-1}^{(j+1)}=M_{k-1}^{(j)}-1
\end{equation}
and in the disappearance of $M_k$.

\begin{proof}[Proof of Proposition \ref{prop:KU}]
First, we give the proof of \eqref{KU} for the case when $K(\omega)>L(\omega)+2$.
We assume that $U(\omega)<K(\omega)$.
As seen above, $K(\omega)=M_m(\omega)$ and we also have that $M_{m+1}(\omega)=K(\omega)-1$.
Furthermore, applying the evolution step $Z$ repeatedly, $M_{m+1}$ follows the bulk behaviour described above,
that is, goes to the left by one in every step as long as $M_{m+1}^{(j)}=M_m^{(j)}-1\ge L^{(j)}+2$.
If this condition holds then we have
$\pi_1((\omega^{(j)})_{M_m^{(j)}-2})=\pi_1((\omega^{(j)})_{M_m^{(j)}-1})=\pi_1((\omega^{(j)})_{M_m^{(j)}})=0$
and $\pi_1((\omega^{(j)})_{M_m^{(j)}+1})=2$.
It corresponds to a projected string $0002$ that moves to the left by one in every step during the bulk behaviour of $M_{m+1}$.
If $U(\omega)\le M_m(\omega)-1=K(\omega)-1$ holds for the initial position of the $1$
then its later position cannot exceed that of the second $0$ of the moving string $0002$ during the bulk phase,
that is, $U^{(j)}\le M_{m+1}^{(j)}$ as long as $M_{m+1}^{(j)}\ge L^{(j)}+2$.

The last step of the bulk behaviour of $M_{m+1}$ happens if $M_{m+1}^{(j)}=L^{(j)}+2$
when the $0002$ string meets the initial string of $2$s of length $L^{(j)}$.
The number of initial $2$s remains the same in this step, the string $0002$ transforms into $0020$
and $M_{m+1}$ annihilates immediately after the end of its bulk behaviour.
If $U(\omega)\le K(\omega)-1$ holds initially then the only possible position of the $1$ after $j+1$ steps is at $L^{(j+1)}+1$,
that is, the string $0020$ must be the projection of $1020$.
Since the $1$ is to the left of all $0$s in $\omega^{(j+1)}$,
after the sorting of the projected string $\Pi_1(\omega^{(j+1)})$, the sorting of $\omega^{(j+1)}$ will also have been completed,
which means that $E_n(\omega)=0$.

Next, we assume that $U(\omega)\ge K(\omega)$.
We show that if $U(\omega)>M_k(\omega)$ for some $k=1,\dots,m$ then the $1$ remains on the right of $M_k$ until the annihilation of $M_k$.
During the bulk behaviour of $M_k$, the $1$ cannot move to the left of $M_k$ trivially.
Assume that $M_k$ arrives at the edge after $j$ steps.
Then we have $M_k^{(j)}=L^{(j)}+1$ and $\pi_1((\omega^{(j)})_{M_k^{(j)}})=0$
but $(\omega^{(j)})_{M_k^{(j)}}=1$ is impossible if $U(\omega)>M_k(\omega)$.
In further steps, the string $02$ at positions $M_k,M_k+1$ transforms into $20$, hence $M_k$ moves to the right by one according to \eqref{Mkedge} and
$(\omega^{(j+1)})_{M_k^{(j+1)}}=0$ also holds.
This means that the $1$ remains on the right of $M_k$ during the edge phase.

The annihilation of $M_k$ happens when $M_k^{(j)}=M_{k-1}^{(j)}-1=L^{(j)}+1$ which corresponds to
$\pi_1((\omega^{(j)})_{L^{(j)}+1})=\pi_1((\omega^{(j)})_{L^{(j)}+2})=0$ and $\pi_1((\omega^{(j)})_{L^{(j)}+3})=2$,
and the projected string $002$ is turned into $020$.
If the $1$ was between $M_k$ and $M_{k-1}$, that is, $U^{(j)}=M_{k-1}^{(j)}$
then the string $012$ at positions $L^{(j)}+1$ to $L^{(j)}+3$ is transformed into $021$.
This means that the $1$ is transferred to the right of $M_{k-1}$ and $U^{(j+1)}=M_{k-1}^{(j+1)}+2=L^{(j+1)}+3$.

If $U(\omega)=K(\omega)$ then $M_{m+1}(\omega)=K(\omega)-1$ which means $U(\omega)=M_m(\omega)$ for the position of $1$
which remains true during the bulk phase of $M_{m+1}$.
The end of the bulk phase means that $U^{(j)}=M_m^{(j)}=M_{m+1}^{(j)}+1=L^{(j)}+3$ for some $j$
and the corresponding string is $0012$ after the initial string of $2$s of length $L^{(j)}$.
In the next step, $M_{m+1}$ annihilates, the string transforms into $0021$ and we have $M_m^{(j+1)}=M_m^{(j)}-1=L^{(j)}+2$ and $U^{(j+1)}=L^{(j)}+4$,
that is, the $1$ is transferred to the right of $M_m$.

By induction, if $U(\omega)\ge K(\omega)$ then the $1$ is transferred to the right of $M_0^{(j)}$ for some $j$.
From then on, the $1$ remains on the right of $M_0$ during its bulk and edge phases as seen above.
The last step of the edge phase of $M_0$ occurs when $M_0^{(j)}=L^{(j)}+1=n-R^{(j)}-1$ for some $j$
after which $\Pi_1(\omega^{(j+1)})$ becomes sorted.
Since the $1$ was to the right of $M_0^{(j)}$, it cannot be at position $L^{(j+1)}+1$ in $\omega^{(j+1)}$ which means that $E_n(\omega)\ge1$.
This completes the proof of \eqref{KU} for the case $K(\omega)>L(\omega)+2$.

If $K(\omega)=L(\omega)+2$, then $U(\omega)<K(\omega)$ necessarily implies that $1$ is to the left of all $0$s already in the initial string $\omega$ and that we have $E_n(\omega)=0$.
In the case $U(\omega)\ge K(\omega)$, we have $M_k(\omega)=L(\omega)+1$ for some integer $k$ and the $1$ is to the right of $M_k$ which results in $E_n(\omega)\ge1$ by the above argument.

Finally, we prove \eqref{Un-R} by showing that $U(\omega)\le n-R(\omega)$ implies $E_n(\omega)\le1$.
If $U(\omega)\le n-R(\omega)$ holds then the last $2$ at position $n-R(\omega)$ initially is swapped with the $1$ in some step $j$,
that is, the string $\omega^{(j)}$ ends with $21$ and a string of $0$s of length $R^{(j)}$.
In any further evolution step after this, the last $2$ can only move to the left by swapping with a $0$.
If this happens then the $1$ is swapped with the same $0$ in the next step.
The distance of the last $2$ and the $1$ cannot exceed $2$,
hence after the last swap of the last $2$, at most one more sorting step is needed, which yields $E_n(\omega)\le1$.
\end{proof}

\section{Limiting excess distribution}
\label{s:asympexcess}

In this section, we prove our main results stated in Theorems \ref{thm:fixp} and \ref{thm:crit}.
The proofs are based on computing the expectation of the switching position $K$.
The key observation is that the difference of the switching position $K$ and the position of the leftmost maximum $M$ can be controlled
in the sense that their expected distance can be upper bounded as in Proposition \ref{prop:KM} below.
The proof of Proposition \ref{prop:KM} is postponed to Section \ref{s:hitting}.

\begin{proposition}\label{prop:KM}
\begin{enumerate}
\item
For $p=1/2-\lambda/(2\sqrt n)$, there is a constant $c$ that depends only on $\lambda$ such that for all $n$ almost surely
\begin{equation}\label{KM}
\E\left(M-K\bigm|M\right)\le c\sqrt{n}.
\end{equation}
\item
For $p<1/2$ fixed, there is a deterministic constant $c$ such that $\E\left(M-K\bigm|M\right)\le c$ almost surely for all $n$.
\end{enumerate}
\end{proposition}

\begin{proposition}\label{prop:EM}
For $p>1/2$, we have that
\begin{equation}\label{EM}
\lim_{n\to\infty}\frac{\E(M)}n=0.
\end{equation}
\end{proposition}

\begin{proof}[Proof of Proposition \ref{prop:EM}]
The regime $p>1/2$ is where the drift of the random walk $S_k$ is $1-2p<0$ negative.
We first show that if $S_0=0$ then for any $\varepsilon>0$, we have that
\begin{equation}\label{maxSn}
\lim_{n\to\infty}\P_0\left(\max_{k\ge\varepsilon\sqrt n}S_k>0\right)=0.
\end{equation}
where the subindex $0$ means conditioning on $\{S_0=0\}$.
Since $m_n=\E_0(S_{\varepsilon\sqrt n})=(1-2p)\varepsilon\sqrt n<0$,
the central limit theorem implies that the fluctuations of $S_{\varepsilon\sqrt n}$ around its mean $m_n$ are of order $n^{1/4}$.
In particular,
\begin{equation}\label{SnCLT}
\lim_{n\to\infty}\P_0\left(S_{\varepsilon\sqrt{n}}>\frac{m_n}2\right)=0.
\end{equation}
By conditioning on the value of $S_{\varepsilon\sqrt n}$, we can write
\begin{equation}\label{maxcond}\begin{aligned}
\P_0\left(\max_{k\ge\varepsilon\sqrt n}S_k>0\right)
&=\P_0\left(\max_{k\ge\varepsilon\sqrt n}S_k>0\Bigm|S_{\varepsilon\sqrt n}>\frac{m_n}2\right)
\P_0\left(S_{\varepsilon\sqrt n}>\frac{m_n}2\right)\\
&\qquad+\P_0\left(\max_{k\ge\varepsilon\sqrt n}S_k>0\Bigm|S_{\varepsilon\sqrt n}<\frac{m_n}2\right)
\P_0\left(S_{\varepsilon\sqrt n}<\frac{m_n}2\right)
\end{aligned}\end{equation}
The first product of probabilities on the right-hand side of \eqref{maxcond} goes to $0$ by \eqref{SnCLT}.
In the second product, the conditional probability can be upper bounded by $a^{|m_n|/2}$
where $a=\P_0(\exists k: S_k=1)<1$ since the drift of $S_k$ is negative.
This proves \eqref{maxSn}.

The limit in \eqref{maxSn} means that
\begin{equation}\label{argmaxto0}
\frac{\argmax_{k\in\{1,\dots,n\}}S_k}{\sqrt{n}}\stackrel\P\to0
\end{equation}
as $n\to\infty$ in probability conditionally given that $S_0=0$.
The argmax may not be unique for the random walk $S_m$ but \eqref{argmaxto0} holds for any of them.
The value of $M$ is equal to the position of the leftmost maximum of the random walk started at $L+1$.
The variable $L$ has a geometric distribution with parameter $1-p$ and in particular it does not depend on $n$.
Hence
\begin{equation}\label{Msqrtn}
\frac M{\sqrt n}\stackrel{\P}\to0
\end{equation}
also holds as $n\to\infty$ in probability.
Since $M\le n$ almost surely, \eqref{Msqrtn} implies \eqref{EM}.
\end{proof}

\begin{proof}[Proof of Theorem \ref{thm:fixp}]
Let first $p>1/2$ be fixed.
By Proposition \ref{prop:KU}, the asymptotic excess distribution can be obtained from the probability
\begin{equation}\label{UKUMbound}
\P(U<K)\le\P(U<M)
\end{equation}
where the upper bound follows because $K\le M$ holds almost surely.

We compute $\P(U<M)$ as follows.
The number of $0$s in the first $m$ bits of the string is equal to $(m+S_m)/2$.
Since $U$ is a uniformly chosen position among the $0$s, for any fixed $m$, we have that
\begin{equation}\label{Um}
\P\left(U<m\bigm|\Pi_1\omega\right)=\frac{m+S_m}{n+S_n}.
\end{equation}
The fluctuations of the random walk $S_m$ with drift $1-2p$ around its mean are of order $\sqrt n$.
More precisely, for any $\varepsilon>0$ there is a $C$ large enough so that
\begin{equation}\label{RWfluct}
\P\left(\max_{m\in\{1,\dots,n\}}|S_m-m(1-2p)|\le C\sqrt n\right)>1-\varepsilon
\end{equation}
for all $n$.
By conditioning, we can write
\begin{equation}\label{U<M}
\P(U<M)=\E\left(\P\left(U<M\bigm|\Pi_1\omega\right)\right)=\E\left(\frac{M+S_M}{n+S_n}\right).
\end{equation}
By \eqref{RWfluct}, the denominator in the expectation on the right-hand side of \eqref{U<M}
is asymptotically equal to $2(1-p)n$ with a correction of order $\sqrt n$ with probability at least $1-\varepsilon$.
On the same event with probability at least $1-\varepsilon$, the numerator is $2(1-p)M$ with a correction of order $\sqrt n$.
Then Proposition \ref{prop:EM} implies that the right-hand side of \eqref{U<M} goes to $0$ as $n\to\infty$.
This together with Proposition \ref{prop:KU} and the bound \eqref{UKUMbound} show that $\lim_{n\to\infty}\P(E_n=0)=0$.
Since $R$ has geometric distribution with parameter $p$ which does not depend on $n$, we also have $\lim_{n\to\infty}\P(E_n>1)=0$
which concludes the proof for $p>1/2$.

For $p<1/2$ fixed, we argue by time reversal and we first show that
\begin{equation}
\lim_{n\to\infty}\frac{\E(M)}n=1.
\end{equation}
It does not directly follow from Proposition \ref{prop:EM} since $M$ is the position of the leftmost maximum
but it is clear from the proof that the statement holds for the position of any maximum.
Then the same argument applies as for $p>1/2$ and in particular \eqref{RWfluct} and \eqref{U<M} hold.
These together show that $\lim_{n\to\infty}\P(U<M)=1$.

By Proposition \ref{prop:KU},
\begin{equation}\label{En0}
\P(E_n=0)=\P(U<K)=\P(U<M)-\P(U\in[K,M))
\end{equation}
where we used in the last equality that $K\le M$ always holds.
By the second part of Proposition \ref{prop:KM}, the expected difference $\E(M-K)$ remains bounded in $n$
hence $\P(U\in[K,M))$ goes to $0$ as $n\to\infty$ which proves that $\lim_{n\to\infty}\P(E_n=0)=1$.

The statement for $p=1/2$ is a special case of Theorem \ref{thm:crit} which is shown below.
\end{proof}

\begin{proof}[Proof of Theorem \ref{thm:crit}]
Let $p=1/2-\lambda/(2\sqrt n)$ with $\lambda\in\R$ fixed.
By Proposition \ref{prop:KU}, the asymptotic excess distribution can be obtained from the probability in \eqref{En0}.
By Proposition \ref{prop:KM}, the expected difference of $K$ and $M$ is of order $\sqrt n$,
hence the probability $\P(U\in[K,M))$ on the right-hand side of \eqref{En0} goes to $0$ as $1/\sqrt n$ in the $n\to\infty$ limit.

We compute $\P(U<M)$ as in \eqref{Um}.
By the fluctuation result \eqref{RWfluct},
the denominator in the expectation on the right-hand side of \eqref{U<M}
is asymptotically equal to $n$ with fluctuations of order $\sqrt n$ with high probability.
The numerator in the same formula is $M$ with a correction of order $\sqrt n$ with high probability.
Hence by boundedness, we have
\begin{equation}
\P(U<M)\sim\frac{\E(M)}n
\end{equation}
as $n\to\infty$.

Since Brownian motion with a drift has a unique argmax almost surely,
Donsker's invariance principle and the continuous mapping theorem together imply that the argmax $M$ of the random walk $S_m$
converges to that of the Brownian motion with drift $\lambda$.
\end{proof}

\section{Hitting time bounds}
\label{s:hitting}

This section is devoted to the proof of Proposition \ref{prop:KM} on the expected distance of $K$ and $M$.
Their expected distance can be bounded by reading the string $\omega$ from $M$ backwards conditionally given $M$.
and by applying hitting time estimates for random walks stated in Lemma \ref{lemma:hittingdistr} below.
Since $M$ is the position of the leftmost maximum,
the distribution of the remaining steps is that of a random walk conditioned not to visit its starting position again.
The trajectory of this walk can then be decomposed into excursions defined by the returns to different levels below the maximum.

To state our random walk estimates, we define $Y_n$ to be a simple random walk starting at $Y_0=0$ and with i.i.d.\ steps
$\P(Y_{n+1}-Y_n=-1)=p$ and $\P(Y_{n+1}-Y_n=1)=q=1-p$ where $p\in(0,1)$.
We let
\begin{equation}\label{defVk}
V_k=\min\{n=0,1,\ldots:Y_n=k\}
\end{equation}
the time of the first visit at level $k$ for any $k\in\Z$.

\begin{lemma}\label{lemma:hittingdistr}
\begin{enumerate}
\item
In the symmetric case $p=1/2$, we have that
\begin{align}
\P_2\left(V_1<m\bigm|V_0>m\right)&\to\frac12\label{hittingconddistr0}\\
\E_2\left(V_1\mathbbm{1}_{\{V_1\leq m\}}\bigm|V_0>m\right)&\sim\sqrt{\frac{\pi}{2}}\sqrt{m}\label{expectedhitting0}
\end{align}
as $m\to\infty$.
\item
If $p=\frac12\left(1-\frac\lambda{\sqrt n}\right)$ then there are constants $c_1>0$ and $c_2$ finite which depend on $\lambda$ such that
\begin{align}
c_1<\P_2\left(V_1<m\bigm|V_0>m\right)&<1-c_1\label{hittingconddistrlambda}\\
\E_2\left(V_1\mathbbm{1}_{\{V_1\leq m\}}\bigm|V_0>m\right)&\le c_2\sqrt{m}\label{expectedhittinglambda}
\end{align}
holds for all $m\le n$.
\end{enumerate}
\end{lemma}

\begin{proof}[Proof of Lemma \ref{lemma:hittingdistr}]
For general $p\in(0,1)$, the distribution of the hitting time is given by
\begin{equation}\label{hittingdistr}
\P_1(V_0=2k+1)=\frac1{2k+1}\binom{2k+1}kp^{k+1}q^k
\end{equation}
for all $k=0,1,\dots$ which can be computed by the reflection principle and it can also be found in \cite{F68}.
The generating function of the probability distribution in \eqref{hittingdistr} is
\begin{equation}
g(s)=\frac{1-\sqrt{1-4pqs^2}}{2qs}.
\end{equation}
The probability that the hitting does not happen in finite time can be given as $\P_1(V_0=\infty)=1-g(1)$
which is equal to $0$ for $p\ge1/2$ and we have $\P_1(V_0=\infty)=(q-p)/q$ for $p<1/2$.

We specify $p=\frac12\left(1-\frac\lambda{\sqrt n}\right)$
and we use the notation $\P^{(\lambda)}$ and $\E^{(\lambda)}$ to denote the dependence on $\lambda$ in this proof
where $\P^{(0)}$ and $\E^{(0)}$ corresponds to the $p=1/2$ case.

By Stirling's approximation, the asymptotics of the probability in \eqref{hittingdistr} is given by
\begin{equation}\label{hittinglambda}
\P_1^{(\lambda)}(V_0=2k+1)\sim\frac1{2\sqrt\pi k^{3/2}}e^{-2\lambda^2k/n}
\end{equation}
as $k\to\infty$.
The tail behaviour of the hitting time can be computed exactly for $p=1/2$ by
\begin{equation}\label{tail0}
\P_1^{(0)}(V_0>m)\sim\sqrt{\frac2\pi}\frac1{\sqrt m}
\end{equation}
as $m\to\infty$ which follows by summation in \eqref{hittinglambda}.
For general $\lambda\in\R$ fixed, we have that
\begin{equation}\label{taillambda}
\frac{\wt c_1}{\sqrt m}\le\P_1^{(\lambda)}(V_0>m)\le\frac{\wt c_2}{\sqrt m}
\end{equation}
in the joint limit when $m,n\to\infty$ provided that $m\le n$.
The constants $\wt c_1,\wt c_2$ are positive and finite and they only depend on $\lambda$.
The upper bound in \eqref{taillambda} follows by the comparison with the symmetric $\lambda=0$ case.
The asymptotic value of $\P_1^{(\lambda)}(V_0=2k+1)$ is upper bounded by $\P_1^{(0)}(V_0=2k+1)$ using \eqref{hittinglambda} for $k$ large.
On the other hand, $\P_1^{(\lambda)}(V_0=\infty)=\lambda/\sqrt n$ for all $\lambda>0$ but this is still $\O(1/\sqrt m)$.
For the lower bound, we sum the right-hand side of \eqref{hittinglambda} between $m$ and $2m$ only
which results in $m$ terms all being at least a constant times $1/m^{3/2}$.

Next we expand the left-hand side of \eqref{hittingconddistr0} and \eqref{EVlambda} in general.
By the definition of the conditional probability and by using the law of total probability, one can write
\begin{multline}\label{PVmlambda}
\P_2^{(\lambda)}\left(V_1<m\bigm|V_0>m\right)\\
=\frac{\sum_{k=0}^{m/2-1}\P_2^{(\lambda)}(V_1=2k+1)\P_1^{(\lambda)}(V_0>m-2k-1)}
{\sum_{k=0}^{m/2-1}\P_2^{(\lambda)}(V_1=2k+1)\P_1^{(\lambda)}(V_0>m-2k-1)+\P_2^{(\lambda)}(V_1>m)}.
\end{multline}
On the other hand,
\begin{equation}\label{EVlambda}\begin{aligned}
\E_2^{(\lambda)}\left(V_1\mathbbm{1}_{\{V_1\leq m\}}\bigm|V_0>m\right)&=\sum_{k=0}^{m/2-1}(2k+1)\P_2^{(\lambda)}\left(V_1=2k+1\bigm|V_0>m\right)\\
&=\sum_{k=0}^{m/2-1}(2k+1)\frac{\P_2^{(\lambda)}(V_1=2k+1)\P_1^{(\lambda)}(V_0>m-2k-1)}{\P_2^{(\lambda)}(V_0>m)}.
\end{aligned}\end{equation}

Next we turn to the proof of \eqref{hittingconddistr0}.
For $\lambda=0$, the asymptotics of the numerator of the right-hand side of \eqref{PVmlambda}
can be obtained by splitting the sum at $m^{1-\varepsilon}$ for some small $\varepsilon>0$.
For the first part of the sum, we have
\begin{equation}\label{sum1stpart}
\sum_{k=0}^{m^{1-\varepsilon}}\P_2^{(0)}(V_1=2k+1)\P_1^{(0)}(V_0>m-2k-1)\sim\sqrt{\frac2\pi}\frac1{\sqrt m}
\end{equation}
because in this regime the probabilities $\P_1^{(0)}(V_0>m-2k-1)\sim\sqrt{\frac2\pi}\frac1{\sqrt m}$ for all $0\le k\le m^{1-\varepsilon}$
with an additive error of order $1/m^{1/2+\varepsilon}$ uniformly.
The sum of the probabilities $\P_2^{(0)}(V_1=2k+1)$ add up to $\P_2^{(0)}(V_1\le2m^{1-\varepsilon}+1)=1-\O(1/m^{1/2-\varepsilon/2})$.

The second part of the sum satisfies
\begin{equation}\label{sum2ndpart}
\sum_{k=m^{1-\varepsilon}}^{m/2-1}\P_2^{(0)}(V_1=2k+1)\P_1^{(0)}(V_0>m-2k-1)=\O\left(\frac1{m^{1-3\varepsilon/2}}\right)
\end{equation}
since we can bound $\P_2^{(0)}(V_1=2k+1)=\O(1/m^{3(1-\varepsilon)/2})$.
We also have
\begin{equation}
\sum_{k=m^{1-\varepsilon}}^{m/2-1}\P_1^{(0)}(V_0>m-2k-1)\sim\sqrt{\frac2\pi}\,\,\sum_{j=1}^{m/2-m^{1-\varepsilon}}\frac1{\sqrt{2j}}=\O(\sqrt m)
\end{equation}
which proves \eqref{sum2ndpart}.
Putting together \eqref{sum1stpart} and \eqref{sum2ndpart} and using it in \eqref{PVmlambda} along with \eqref{tail0} yields \eqref{hittingconddistr0}.

Using the asymptotics \eqref{hittinglambda} with $\lambda=0$ and \eqref{tail0} in \eqref{EVlambda}, we get
\begin{equation}\begin{aligned}
\E_2^{(\lambda)}\left(V_1\mathbbm{1}_{\{V_1\leq m\}}\bigm|V_0>m\right)
&\sim\sum_{k=0}^{m/2-1}(2k+1)\frac{\frac1{2\sqrt\pi k^{3/2}}\cdot\sqrt{\frac2\pi}\frac1{\sqrt{m-2k-1}}}{\sqrt{\frac2\pi}\frac1{\sqrt m}}\\
&\sim\frac1{\sqrt\pi}\sqrt m\frac1m\sum_{k=0}^{m/2-1}\frac1{\sqrt{\frac km}}\frac1{\sqrt{1-\frac{2k+1}m}}\\
&\sim\sqrt{\frac\pi2}\sqrt m
\end{aligned}\end{equation}
since
\begin{equation}
\frac1m\sum_{k=0}^{m/2-1}\frac1{\sqrt{\frac km}}\frac1{\sqrt{1-\frac{2k+1}m}}\to\int_0^{1/2}\frac1{\sqrt x}\frac1{\sqrt{1-2x}}\,\d x=\frac\pi{\sqrt2}
\end{equation}
as $m\to\infty$ which proves \eqref{expectedhitting0}.

The proofs of \eqref{hittingconddistrlambda} and \eqref{expectedhittinglambda} follow by comparison with their $\lambda=0$ counterparts
\eqref{hittingconddistr0} and \eqref{expectedhittinglambda}.
The input for proving \eqref{hittingconddistr0} and \eqref{expectedhitting0} are the asymptotics of the hitting time distribution
given in \eqref{hittinglambda} with $\lambda=0$ and in \eqref{tail0}.
These asymptotics are then used in \eqref{PVmlambda} and \eqref{EVlambda}.
Let $\lambda\in\R$ be fixed.
The expansions \eqref{PVmlambda} and \eqref{EVlambda} are still valid.
The asymptotic expansion \eqref{hittinglambda} is now used with the given value of $\lambda$.
Since we assume that $m\le n$, the asymptotics only change by constant factors.
More precisely, instead of the asymptotics in \eqref{hittinglambda}, we can give upper and lower bounds with two different positive and finite constant prefactors.
Similarly, instead of the asymptotics \eqref{tail0}, the bounds given in \eqref{taillambda} are available for the tail of the hitting time.
Using these asymptotic results, the bounds in \eqref{hittingconddistrlambda} and \eqref{expectedhittinglambda} follow in the same way as the asymptotics in the $\lambda=0$ case.
\end{proof}

\begin{proof}[Proof of Proposition \ref{prop:KM}]
We start with a detailed proof of the first part about the $p=1/2-\lambda/(2\sqrt n)$ case.
In the end, we show how the statement for fixed $p<1/2$ follows.

Conditionally given $M$ and $L$, we introduce the reversed and reflected random walk $\wt S_i$ seen from the point $(M,S_M)$ given by
\begin{equation}\label{defStilde}
\wt S_i=S_M-S_{M-i}
\end{equation}
for $i=0,1,\dots,M-L$.
By the definition \eqref{defStilde} and by the leftmost maximum property of $M$,
it follows that the law of $\wt S_i$ is that of a simple random walk with drift $2p-1$ starting at $\wt S_0=0$
conditioned not to return to $0$ until the $(M-L)$th step.
In particular we have $S_1=1$ and $S_2=2$.

Next we introduce a sequence of indicators $I_k$ and the corresponding random times $\tau_k$.
We let $\tau_0=1$ and we let $I_0=\ind(\wt S_{\tau_0+i}>\wt S_{\tau_0}-1\mbox{ for }i=1,\dots,M-L-\tau_0)$
which is almost surely equal to $1$ since $\wt S_i$ cannot return to $0$.
Next we define $\tau_1=\tau_0+1=2$.
In general, given $\tau_k$, we let
\begin{equation}
I_k=\ind\left(\wt S_{\tau_k+i}>\wt S_{\tau_k}-1\mbox{ for }i=1,\dots,M-L-\tau_k\right)
\end{equation}
for $k=1,2,\dots$, that is, $I_k$ is the indicator that the random walk $(\wt S_{\tau_k+i},i=1,2,\dots,M-L-\tau_k)$
stays strictly above the level $\wt S_{\tau_k}-1$.

If $I_k=1$ then we let $\tau_{k+1}=\tau_k+1$ and in this case we have $\wt S_{\tau_{k+1}}=\wt S_{\tau_k}+1$.
If $I_k=0$ then we define
\begin{equation}\label{deftauk+1}
\tau_{k+1}=\min\left\{i\in\{1,\dots,M-L-\tau_k\}:\wt S_{\tau_k+i}=\wt S_{\tau_k}-1\right\}+1
\end{equation}
where the minimum above is the time of the first return to the level $\wt S_{\tau_k}-1$ which is finite by the fact that $I_k=0$.
The addition of $1$ on the right-hand side of \eqref{deftauk+1} ensures that $\wt S_{\tau_{k+1}}=\wt S_{\tau_k}$.
We let the $k$th excursion last between times $\tau_k$ and $\tau_{k+1}$ with length $W_k=\tau_{k+1}-\tau_k$.

The pair $(\tau_k,I_k)$ forms a Markov chain and its transition probabilities are described below.
For deterministic $t_k,t_{k+1}\in\Z_+$ and $i_k,i_{k+1}\in\{0,1\}$, one can write the conditional probability
\begin{equation}\label{transprob}
P((t_k,i_k),(t_{k+1},i_{k+1}))=\P\left((\tau_{k+1},I_{k+1})=(t_{k+1},i_{k+1})\bigm|(\tau_k,I_k)=(t_k,i_k)\right)
\end{equation}
depending on $i_k$ as follows.
If $i_k=1$ then the transition probability $P((t_k,1),(t_{k+1},i_{k+1}))$ is nonzero only if $t_{k+1}=t_k+1$.
In that case we have
\begin{equation}\begin{aligned}
P((t_k,1),(t_k+1,1))&=1-P((t_k,1),(t_k+1,0))\\
&=\P_2\left(V_1>M-L-t_{k+1}\bigm|V_0>M-L-t_{k+1}\right)
\end{aligned}\end{equation}
since between times $\tau_{k+1}$ and $M-L$, the walk $\wt S_j$ is conditioned not to go strictly below $\wt S_{\tau_{k+1}}-1$.

If $i_k=0$ then the distribution of $\tau_{k+1}$ is given by
\begin{equation}\label{ik0distrtau}
P((t_k,0),(t_{k+1},\{0,1\}))=\P_2\left(V_1=t_{k+1}-t_k\bigm|V_1\le M-L-t_k<V_0\right)
\end{equation}
since the random walk $\wt S_j$ conditioned not to go strictly below $\wt S_{\tau_k}-1$ between times $\tau_k$ and $M-L$
does visit $\wt S_{\tau_k}-1$ and the time of the first visit is $\tau_{k+1}$
with distribution given by the right-hand side of \eqref{ik0distrtau}.
Further, given that $\tau_{k+1}=t_{k+1}$, we have
\begin{equation}\label{ik0distri}\begin{aligned}
\P\left(I_{k+1}=1\bigm|\tau_{k+1}=t_{k+1}\right)
&=1-\P\left(I_{k+1}=0\bigm|\tau_{k+1}=t_{k+1}\right)\\
&=\P_2\left(V_1>M-L-t_{k+1}\bigm|V_0>M-L-t_{k+1}\right)
\end{aligned}\end{equation}
because the walk $\wt S_j$ between $\tau_{k+1}$ and $M-L$ is conditioned not to go strictly below $\wt S_{\tau_{k+1}}-1$.
Then the transition probabilities $P((t_k,0),(t_{k+1},i_{k+1}))$ are given by the products of the probabilities
in \eqref{ik0distrtau} and \eqref{ik0distri}.

With the definition $W_k=\tau_{k+1}-\tau_k$ being the length of the $k$th excursion, we let
\begin{equation}
N=\min\{k=1,2,\ldots:I_k=I_{k-1}=1\}
\end{equation}
to be the first time when there are two consecutive trivial excursions which is a stopping time.
Then we can write
\begin{equation}\label{M-K}
M-K+1=\sum_{k=1}^NW_k
\end{equation}
by the definition of $K$.
Since $N$ is a stopping time for the time inhomogeneous Markov chain $(\tau_k,I_k)$ and for its natural filtration $\mathcal F_k$,
the computation similar to the proof of Wald's identity yields
\begin{equation}\begin{aligned}
\E\left(\sum_{k=1}^{N}W_k\right)
&=\sum_{k=1}^{\infty}\E\left(W_k\mathbbm{1}_{\{N\geq k\}}\right)\\
&=\sum_{k=1}^{\infty}\E\left(\E\left(W_k\mathbbm{1}_{\{N\geq k\}}\bigm|\mathcal F_{k-1}\right)\right)\\
&=\sum_{k=1}^{\infty}\E\left(\E\left(W_k\bigm|\mathcal F_{k-1}\right)\mathbbm1_{\{N\ge k\}}\right).
\end{aligned}\end{equation}
The conditional expectation $\E(W_k|\mathcal F_{k-1})$ is bounded by a constant multiple of $\sqrt n$ uniformly in $k$
by Lemma \ref{lemma:hittingdistr}.
The same lemma also implies that $N$ has an exponential tail and in particular $\E(N)$ is bounded by a constant.
Hence the expectation of \eqref{M-K} is upper bounded by a constant times $\sqrt n$ which completes the proof of \eqref{KM} and that of the first part.

If $p<1/2$ is fixed then the same decomposition into excursions of length $W_k$ remain valid.
The difference is that the number and the expected lengths of excursions is now related to those of a random walk with a fixed positive drift.
Instead of Lemma \ref{lemma:hittingdistr} we have that $\P(V_0=\infty)>0$ and $V_1$ has an exponential tail.
Hence we have that
\begin{equation}
\lim_{m\to\infty}\P_2\left(V_1<m\bigm|V_0>m\right)=\P_2\left(V_1<\infty\bigm|V_0=\infty\right)=\frac{\frac pq\frac{q-p}q}{1-\left(\frac pq\right)^2}=p.
\end{equation}
On the other hand, the conditional expectation $\E_2\left(V_1\mathbbm{1}_{\{V_1\leq m\}}\bigm|V_0>m\right)$ is bounded by a constant uniformly in $m$.
These bounds are enough to conclude the second part of the proposition.
\end{proof}

\bibliography{3type_seq}
\bibliographystyle{alpha}

\end{document}